\documentclass{article}
\usepackage{verbatim,amsmath,amsfonts,amssymb,amsthm}

\newtheorem {theorem}{Theorem}[section]
\newtheorem {proposition}[theorem]{Proposition}
\newtheorem {lemma}[theorem]{Lemma}
\newtheorem {example}[theorem]{Example}
\newtheorem {corollary}[theorem]{Corollary}
\newtheorem {question}[theorem]{Question}
\newtheorem {remark}[theorem]{Remark}
\newtheorem*{idfn}{Definition}

\numberwithin{equation}{section}

\begin{document}

\title{On perfect cones and absolute Baire-one retracts}

\author{Olena Karlova}

\maketitle

{\small
  We introduce perfect cones over topological spaces and study their connection with absolute $B_1$-retracts.}

\section{Introduction}

A subset $E$ of a topological space $X$ is a retract of $X$ if there exists a continuous mapping $r:X\to E$ such that $r(x)=x$ for all $x\in E$. Different modifications of this notion in which $r$ is allowed to be discontinuous (in particular, almost continuous or a Darboux function) were considered in \cite{Garrett,Paw,RGK,St}. The author introduced in \cite{KRaex} the notion of $B_1$-retract, i.e. a subspace $E$ of $X$ for which there exists a Baire-one mapping  $r:X\to E$ with $r(x)=x$ on $E$. Moreover, the following two results were obtained in \cite{KRaex}.

\begin{theorem}\label{th:Raex1}
Let $X$ be a normal space and $E$ be an arcwise connected and locally arcwise connected metrizable $F_\sigma$- and $G_\delta$-subspace of $X$. If

(i) $E$ is separable, or

(ii) $X$ is collectionwise  normal,

\noindent then $E$ is a $B_1$-retract of $X$.
\end{theorem}

\begin{theorem} Let $X$ be a completely metrizable space and let $E$ be an arcwise connected and locally arcwise connected $G_\delta$-subspace of $X$. Then $E$ is a $B_1$-retract of $X$.
\end{theorem}

Note that in the above mentioned results $E$ is a locally arcwise connected space. Therefore, it is naturally to ask

\begin{question}\label{question}
 Is any arcwise connected $G_\delta$-subspace $E$ of a completely metrizable space $X$ a $B_1$-retract of this space?
\end{question}

In this work we introduce the notions of the perfect cone over a topological space and an absolute $B_1$-retract~(see definitions in Section~\ref{sec:Preliminaries}). We show that the perfect cone over a $\sigma$-compact metrizable zero-dimensional space is an absolute $B_1$-retract. Moreover, we give the negative answer to Question~\ref{question}.

\section{Preliminaries}\label{sec:Preliminaries}

Throughout the paper, all topological spaces have no separation axioms if it is not specified.

A mapping $f:X\to Y$ is {\it a Baire-one mapping} if there exists a sequence of continuous mappings $f_n:X\to Y$ which converges to $f$ pointwise on $X$.

A subset $E$ of a topological space $X$ is called
\begin{itemize}

\item  {\it a $B_1$-retract of $X$} if there exists a sequence of continuous mappings $r_n:X\to E$ such that $r_n(x)\to r(x)$ for all $x\in X$ and $r(x)=x$ for all $x\in E$; the mapping $r:X\to E$ is called {\it a $B_1$-retraction} of $X$ onto~$E$;

\item {\it a $\sigma$-retract of $X$} if $E=\bigcup\limits_{n=1}^\infty E_n$, where $(E_n)_{n=1}^\infty$ is an increasing sequence of retracts of  $X$;

\item {\it ambiguous} if it is simultaneously $F_\sigma$ and $G_\delta$ in $X$.
\end{itemize}

A topological space $X$ is
\begin{itemize}

\item {\it perfectly normal} if it is normal and every closed subset of $X$ is $G_\delta$;

 \item {\it an absolute $B_1$-retract (in the class $\mathcal C$ of topological spaces)} if $X\in\mathcal C$ and for any homeomorphism $h$, which maps $X$ onto a $G_\delta$-subset $h(X)$ of a space $Y\in\mathcal C$, the set $h(X)$ is a $B_1$-retract of $Y$; in this paper we will consider only the case $\mathcal C$ is the class of all perfectly normal spaces;

  \item {\it a space with the regular $G_\delta$-diagonal} if there exists a sequence $(G_n)_{n=1}^\infty$ of open neighborhoods of the diagonal $\Delta=\{(x,x):x\in X\}$ in $X^2$ such that  $\Delta=\bigcap\limits_{n=1}^\infty G_n=\bigcap\limits_{n=1}^\infty  \overline{G}_n$;

  \item {\it contractible} if there exists a continuous mapping $\gamma:X\times [0,1]\to X$ and $x^*\in X$ such that $\gamma(x,0)=x$ and $\gamma(x,1)=x^*$ for all $x\in X$; the mapping $\gamma$ we call {\it a contraction}.
\end{itemize}

A family $(A_s:s\in S)$ of subsets of $X$ is said to be {\it a partition of $X$} if $X=\bigcup\limits_{s\in S}A_s$ and $A_s\cap A_t=\O$ for all $s\ne t$.

If a space $X$ is homeomorphic to a space $Y$ then  we denote this fact by $X\simeq Y$.

For a mapping $f:X\times Y\to Z$ an a point $(x,y)\in X\times Y$ let $f^x(y)=f_y(x)=f(x,y)$.

\section{Perfect cones and their properties}
{\it The cone} $\Delta (X)$ {\it over a topological space} $X$ is the quotient space $(X\times [0,1])/(X\times\{0\})$ with the quotient mapping $\lambda:X\times [0,1]\to \Delta(X)$. By $v$ we denote {\it the vertex of the cone}, i.e. $v=\lambda(X\times\{0\})$. We call the set $\lambda(X\times\{1\})$ {\it the base of the cone}.

Let $X_1=(0,1)$ and $X_2=[0,1]$. Then $\Delta(X_2)$ is homeomorphic to a triangle $T\subseteq [0,1]^2$, while $\Delta(X_1)$ is not even metrizable, since there is no countable base of neighborhoods at the cone vertex. Consequently, the naturally embedding of $\Delta(X_1)$ into $\Delta(X_2)$ is not a homeomorphism. Therefore, on the cone $\Delta(X)$ over a space $X$ naturally appears one more  topology $\mathcal T_p$ which coincides with the quotient topology $\mathcal T$ on $X\times (0,1]$ and the base of neighborhoods of the vertex $v$ forms the system $\{\lambda(X\times [0,\varepsilon)):\varepsilon>0\}$. The cone $\Delta(X)$ equipped with the topology $\mathcal T_p$ is said to be {\it perfect} and is denoted by $\Delta_p(X)$.

For all $x\in X$ we write
$$
vx=\lambda^x([0,1]).
$$
Obviously,
$$
\Delta(X)=\bigcup\limits_{x\in X}vx.
$$
It is easily seen that $vx\cap vy=\{v\}$ for all distinct $x,y\in X$ and $vx\simeq [0,1]$ for all $x\in X$.

For every $y\in \Delta(X)\setminus \{v\}$ we set
\begin{equation}\label{eq:alpha}
\alpha(y)={\rm pr}_X(\lambda^{-1}(y)).
\end{equation}
Obviously, $\alpha:\Delta(X)\setminus \{v\}\to X$ is continuous in both topologies $\mathcal T$ and $\mathcal T_p$.

Let
\begin{equation}\label{eq:beta}
\beta(y)=\left\{\begin{array}{ll}
                  {\rm pr}_{[0,1]}(\lambda^{-1}(y)), & y\ne v, \\
                  0, & y=v.
                \end{array}
\right.
\end{equation}
Then $\beta:\Delta_p(X)\to [0,1]$ is a continuous function.  Indeed, it is evident that $\beta$ is continuous on $\Delta_p(X)\setminus\{v\}$. Since
\begin{equation}\label{eq:betabase}
\beta^{-1}([0,\varepsilon))=\lambda(X\times [0,\varepsilon))
\end{equation}
for any $\varepsilon>0$, the set $\beta^{-1}([0,\varepsilon))$ is a neighborhood of $v$. Consequently, $\beta$ is continuous at $v$.

We observe that
$$
\lambda(\alpha(y),\beta(y))=y
$$
for all $y\in \Delta(X)\setminus\{v\}$.

\begin{remark}\label{rem} {\rm
\begin{enumerate}
\item The concept of the perfect cone over a separable metrizable space was also defined in~\cite[p.~55]{JVM}.

\item\label{rem:subspace} We observe that $x\mapsto \lambda(x,1)$ is a homeomorphism of $X$ onto $\lambda(X,\times\{1\})\subseteq \Delta_p(X)$. Therefore, we can identify $X$ with its image and consider $X$ as a subspace of $\Delta_p(X)$.

\item\label{rem:retract} In the light of the previous observation we may assume that the mapping $\alpha$ defined by formula~(\ref{eq:alpha}) is a retraction.

\item\label{rem:betabase} The system $\{\beta^{-1}([0,\varepsilon)):\varepsilon>0\}$ is the base of neighborhoods of the vertex of the cone according to~(\ref{eq:betabase}).
\end{enumerate}
 }
\end{remark}

\begin{proposition}\label{prop:CompactPerfectCone}{\rm (cf. \cite[p.~55]{JVM})}
   The cone $\Delta(X)$ over a compact space $X$ is perfect.
\end{proposition}

\begin{proof}
Let $W$ be an open neighborhood of $v$ in  $\Delta(X)$. Then for every $x\in X$ there exist a neighborhood $U_x$ of  $x$ and $\delta_x>0$ such that $\lambda(U_x\times [0,\delta_x))\subseteq W$. Choose a finite subcover $(U_1,\dots,U_n)$ of $(U_x:x\in X)$ and put $\varepsilon=\min\{\delta_1,\dots,\delta_n\}$. Then $\lambda(X\times [0,\varepsilon))\subseteq W$. Hence, $\Delta(X)$ is the perfect cone.
\end{proof}

\begin{theorem}\label{th:PropertiesOfCones} Let $X$ be a topological space.
\begin{enumerate}
   \item If $X$ is Hausdorff, then $\Delta_p(X)$ is Hausdorff.

    \item If $X$ is regular, then $\Delta_p(X)$ is regular.

   \item\label{PerfectlyNormal} If  $X$ is a countable regular space, then $\Delta_p(X)$ is perfectly normal.

   \item\label{contract} $\Delta_p(X)$  is contractible.

  \item\label{LocConn} If $X$ is locally (arcwise) connected, then: a) $\Delta(X)$ is locally (arcwise) connected;
  b)  $\Delta_p(X)$ is locally (arcwise) connected.

  \item\label{metrizability} If $X$ is metrizable, then $\Delta_p(X)$ is metrizable.
\end{enumerate}
\end{theorem}

\begin{proof} 1). Let $x,y\in \Delta_p(X)$ and $x\ne y$. Since $\Delta_p(X)\setminus\{v\}$ is homeomorphic to the Hausdorff space $X\times (0,1]$, it is sufficient to consider the case $x=v$ or $y=v$. Assume that  $x=v$ and $y\ne v$. Then \mbox{$0=\beta(x)<\beta(y)\le 1$}, where $\beta$ is defined by formula (\ref{eq:beta}). Set $O_x=\lambda(X\times [0,\beta(y)/2))$ and $O_y=\lambda(X\times (\beta(y)/2,1])$.
Then $O_x$ and $O_y$ are disjoint neighborhoods of $x$ and $y$ in  $\Delta_p(X)$, respectively.
\medskip

2). Fix $y\in Y$ and a closed set $F\subseteq \Delta_p(X)$ such that  $y\not\in F$. Since $X\times (0,1]$ is regular, the case $y\ne v$ and $v\not\in F$ is obvious.

Let $y=v$. Choose $\varepsilon>0$ such that $F\cap\lambda(X\times [0,\varepsilon))=\O$. Then $U=\lambda(X\times [0,\varepsilon/2))$ and $V=\lambda(X\times (\varepsilon/2,1])$ are disjoint open neighborhoods of $v$ and $F$ in $\Delta_p(X)$, respectively.

Now let $y\ne v$ and $v\in F$. We take $\varepsilon>0$ such that $O_v=\beta^{-1}([0,\varepsilon))$ is an open neighborhood of $v$ with $y\not\in\overline O_v$. Moreover, since $\Delta_p(X)\setminus \{v\}$ is regular, there exists an open neighborhood $O_y$ of $y$ in $\Delta_p(X)$ with $\overline{O_y}\cap G\subseteq \Delta_p\setminus F$. Then $U=O_y\setminus\overline{O_v}$ is an open neighborhood of $y$ in $\Delta_p(X)$ such that $\overline U\cap F=\emptyset$.

\medskip

3). We observe that $\Delta_p(X)$ is regular by the previous proposition. Moreover,  $Y$ is hereditarily Lindel\"{o}f (and, consequently, normal) as the union of countably many homeomorphic copies of $[0,1]$. Hence, $\Delta_p(X)$ is perfectly normal.

\medskip

4). For all $y\in \Delta_p(X)$ and  $t\in [0,1]$ define
\begin{equation}\label{eq:gamma}
\gamma(y,t)=\left\{\begin{array}{ll}
                     \lambda(\alpha(y), t\cdot \beta(y)), & y\ne v, \\
                     v, & y=v,
                   \end{array}
\right.
\end{equation}
where $\alpha$ and $\beta$ are defined by~(\ref{eq:alpha}) and (\ref{eq:beta}), respectively. Then $\gamma(y,0)=v$ and $\gamma(y,1)=y$ for all $y\in Y$.
Clearly, $\gamma$ is continuous on $(Y\setminus\{v\})\times [0,1]$. Let $\varepsilon>0$ and $W=\lambda(X\times [0,\varepsilon))=\beta^{-1}([0,\varepsilon))$. Since $\beta(\gamma(y,t))=t\cdot\beta(y)$, $\gamma(W\times [0,1])\subseteq W$. Hence, $\gamma$ is continuous at each point of the set $\{v\}\times [0,1]=\beta^{-1}(0)$.
\medskip

5a). Notice that the space $\Delta(X)\setminus\{v\}$ is locally (arcwise) connected, since it is homeomorphic to the locally (arcwise) connected space $X\times (0,1]$. Let us check that $\Delta(X)$ is locally (arcwise) connected at $v$. Fix an open neighborhood $W$ of $v$ in $\Delta(X)$. Then $V=\lambda^{-1}(W)$ is open in $X\times [0,1]$ and $X\times \{0\}\subseteq V$. For every $x\in X$ we denote by $G_x$ the (arcwise) component of $V$ with $(x,0)\in G_x$. Then $G_x$ is open in $\Delta(X)\setminus\{v\}$ and, consequently, in  $\Delta(X)$. Let $G=\bigcup\limits_{x\in X}G_x$. Then the set $\lambda(G)$ in open neighborhood of $v$ such that $\lambda(G)\subseteq W$. It remains to observe that $\lambda(G)$ is (arcwise) connected, since $v\in\lambda(G_x)$ and  $\lambda(G_x)$ is (arcwise) connected for all $x\in X$.
\medskip

5b). Let $\varepsilon>0$ and $W=\beta^{-1}([0,\varepsilon))$. Since each element of $W$ can be joined by a segment with $v$, $W$ is an arcwise connected neighborhood of $v$.
\medskip

6). Let $\varrho$ be a metric generating the topology of $X$ with $\varrho\le 1$. For all $x,y\in X$ and $s,t\in[0,1]$ we set
$$
d(\lambda(x,s),\lambda(y,t))=|t-s|+\min\{s,t\}\varrho(x,y).
$$
Then $d$ is a correctly defined, symmetric, nonnegative and nondegenerate mapping of $\Delta_p(X)\times\Delta_p(X)$. Moreover, the triangle inequality is satisfied, i.e.
$$
d(\lambda(x,s),\lambda(z,u))\le d(\lambda(x,s),\lambda(y,t))+d(\lambda(y,t),\lambda(z,u))
$$
for all $(x,s), (y,t), (z,u)\in X\times [0,1]$. If $t\ge\min\{s,u\}$, the inequality is obvious. Let $t<\min\{s,u\}$. Without loss of generality, we may assume that $t<u\le s$. Then the above inequality is equivalent to
$$
s-u+u\varrho(x,z)\le s-t+t\varrho(x,y)+u-t+t\varrho(y,z),
$$
i.e.,
$$
\varrho(x,z)\le 2(1-\frac tu)+\frac tu(\varrho(x,y)+\varrho(y,z)),
$$
which does hold since $\varrho(x,z)\le 2$ and $\varrho(x,z)\le \varrho(x,y)+\varrho(y,z)$. Moreover, $d$ generates the topology of the perfect cone. It is obvious that the $d$-neighborhoods of $v$ are the correct ones and the metric $d$ on $\lambda(X\times (0,1])$ is equivalent to the summing metric inherited from $X\times(0,1]$.
\end{proof}

A subset $E$ of a topological vector space $X$ is {\it bounded} if for any neighborhood of zero $U$ there is such $\gamma>0$ that $E\subseteq\delta U$ for all $|\delta|\ge\gamma$.

\begin{proposition}\label{prop:emb}
  Let $Z$ be a topological vector space and $X\subseteq Z$ be a bounded set. Then $\Delta_p(X)$ is embedded to $Z\times\mathbb R$.
\end{proposition}

\begin{proof}
  Consider the set $C=\{(xt,t):x\in X, t\in [0,1]\}$. Let $\varphi(x,t)=(xt,t)$ for all $(x,t)\in X\times [0,1]$ and $v^*=(0,0)\in X\times [0,1]$. Then the restriction $\varphi|_{X\times (0,1]}$ is a homeomorphism onto $C\setminus\{v^*\}$. Moreover, the mapping $\beta:C\to [0,1]$, $\beta(z,t)=t$, is continuous. Therefore, $\beta^{-1}([0,\varepsilon))$ is an open neighborhood of $v^*$ in $C$ for any $\varepsilon>0$. Now we show that the system $\{\beta^{-1}([0,\varepsilon)):\varepsilon>0\}$ is a base of $v^*$. Take an open neighborhood of zero in $Z$,  $\delta>0$  and let $W=U\times (-\delta,\delta)$. Choose $\varepsilon\in (0,\delta)$ such that $tX\subseteq U$ for all $t$ with $|t|<\varepsilon$. Then for each $y=(z,t)\in\beta^{-1}([0,\varepsilon))$ we have $|t|<\delta$ and $z\in tx\subseteq U$. Consequently, $\beta^{-1}([0,\varepsilon))\subseteq W$. Hence, $C$ is homeomorphic to the perfect cone $\Delta_p(X)$.
\end{proof}

The following result easily follows from \cite[Theorem 1.5.9]{JVM}.

\begin{corollary}\label{th:ConeOverFiniteSet}
  The cone $\Delta_p(X)$ over a finite Hausdorff space $X$ is an absolute retract.
\end{corollary}

\section{Weak $B_1$-retracts}

A subset $E$ of a topological space  $X$ is called {\it a weak $B_1$-retract} of $X$ if there exists a sequence of continuous mappings $r_n:X\to E$ such that $\lim\limits_{n\to\infty}r_n(x)=x$ for all $x\in E$. Clearly, every $B_1$-retract is a weak $B_1$-retract. The converse proposition is not true (see Example~\ref{ex:weakB1}).

A space $X$ is called {\it an absolute weak $B_1$-retract} if for any space $Y$ and for any homeomorphic embedding $h:X\to Y$ the set $h(X)$ is a weak $B_1$-retract of~$Y$.

Let $E=\bigcup\limits_{n=1}^\infty E_n$ and let $(r_n)_{n=1}^\infty$ be a sequence of retractions $r_n:X\to E_n$. If the sequence $(E_n)_{n=1}^\infty$ is increasing then $\lim\limits_{n\to\infty}r_n(x)=x$ for every $x\in E$. Thus, have proved the following fact.

\begin{proposition}\label{prop:SigmaRetract}
  Every $\sigma$-retract of a topological space $X$ is a weak $B_1$-retract of $X$.
\end{proposition}

\begin{proposition}\label{prop:weak}
 Let $X$ be a countable Hausdorff space. Then $\Delta_p(X)$ is an absolute weak $B_1$-retract.
\end{proposition}

\begin{proof} Assume that $\Delta_p(X)$ is a subspace of a topological space $Z$.
   Let $X=\{x_n:n\in\mathbb N\}$ and $X_n=\{x_1,\dots,x_n\}$. Then $\Delta_p(X)=\bigcup\limits_{n=1}^\infty\Delta_p(X_n)$ and every $\Delta_p(X_n)$ is a retract of $Z$ by Corollary~\ref{th:ConeOverFiniteSet}. Then $\Delta_p(X)$ is a weak $B_1$-retract of $Z$ by  Proposition~\ref{prop:SigmaRetract}.
\end{proof}

It was proved in~\cite{KRaex} that a $B_1$-retract of a connected space is connected. It turns out that this is still valid for weak  $B_1$-retracts.

\begin{theorem}
  Let $X$ be a connected space. Then any weak $B_1$-retract $E$ of $X$ is connected.
\end{theorem}

\begin{proof}
 Let $(r_n)_{n=1}^\infty$ be a sequence of continuous mappings $r_n:X\to E$ such that $\lim\limits_{n\to\infty}r_n(x)=x$ for all $x\in E$.
Denote $H=\bigcup\limits_{n=1}^\infty r_n(X)$. We show that $H$ is connected. Conversely, suppose that $H=H_1\cup H_2$, where $H_1$ and $H_2$ are disjoint sets which are closed in $H$. Observe that each set $B_n=r_n(X)$ is connected. Then $B_n\subseteq H_1$ or $B_n\subseteq H_2$. Choose an arbitrary $x\in H_1$. Then there exists a number $n_1$ such that $r_n(x)\in H_1$ for all $n\ge n_1$. Hence, $B_n\subseteq H_1$ for all $n\ge n_1$. Similarly, there exists a number  $n_2$ such that $B_n\subseteq H_2$ for all $n\ge n_2$. Therefore,  $B_n\subseteq H_1\cap H_2$ for all $n\ge \max\{n_1,n_2\}$, which is impossible.

It is easy to see that $H\subseteq E\subseteq \overline{H}$. Since $H$ and $\overline{H}$ are connected, $E$ is connected too.
\end{proof}

\begin{lemma}\label{l1} Let $X$ be a normal space, $Y$ be a contractible space, $(F_i)_{i=1}^n$ be a sequence of disjoint closed subsets of $X$ and let $g_i:X\to Y$ be a continuous mapping for every $1\le i\le n$. Then there exists a continuous mapping \mbox{$g:X\to Y$} such that
$g(x)=g_i(x)$ on $F_i$ for every $1\le i\le n$.
\end{lemma}

\begin{proof} Let $y^*\in Y$ and $\gamma:Y\times [0,1]\to Y$ be a continuous mapping such that $\gamma(y,0)=y$ and $\gamma(y,1)=y^*$ for all $y\in Y$.
 For all $x,y\in Y$ and $t\in[0,1]$ define
$$
h(x,y,t)=\left\{\begin{array}{ll}
                  \gamma(x,2t), & 0\le t\le 1/2, \\
                  \gamma(y,-2t+2), & 1/2<t\le 1.
                \end{array}
\right.
$$
Then the mapping $h:Y\times Y\times [0,1]\to Y$ is continuous, $h(x,y,0)=x$ and $h(x,y,1)=y$.

Let $n=2$. By Urysohn's Lemma there is a continuous function  $\varphi:X\to
[0,1]$ such that $\varphi(x)=0$ on $F_1$ and $\varphi(x)=1$ on $F_2$. For all $x\in X$ let $$g(x)=h(g_1(x),g_2(x),\varphi(x)).$$ Clearly, $g:X\to Y$ is continuous and  $g(x)=g_1(x)$ if $x\in F_1$, and $g(x)=g_2(x)$ if $x\in F_2$.

Assume the assertion of the lemma is true for $k$ sets, where  $k=1,\dots,n-1$, and prove it for $n$ sets. According to our assumption, there exists a continuous mapping $\tilde{g}:X\to Y$ such that $\tilde{g}|_{F_i}=g_i$ for every $i=1,\dots,n-1$. Since the sets
$F=\bigcup\limits_{i=1}^{n-1}F_i$ and $F_n$ are closed and disjoint, there exists a continuous mapping  $g:X\to Y$ such that $g|_F=\tilde{g}$ and
$g|_{F_n}=g_n$. Then $g|_{F_i}=g_i$ for every $1\le i\le n$.
\end{proof}

\begin{theorem}\label{th:ContrWeakB1}
  Let $E$ be a contractible ambiguous weak $B_1$-retract of a normal space $X$. Then $E$ is a $B_1$-retract of $X$.
\end{theorem}

\begin{proof}
  Let $(r_n)_{n=1}^\infty$ be a sequence of continuous mappings $r_n:X\to E$ such that $\lim\limits_{n\to\infty}r_n(x)=x$ for all $x\in E$.
  Choose increasing sequences $(E_n)_{n=1}^\infty$ and $(F_n)_{n=1}^\infty$ of closed subsets of $X$ such that $E=\bigcup\limits_{n=1}^\infty E_n$ and $X\setminus E=\bigcup\limits_{n=1}^\infty F_n$. Fix  $x^*\in E$. Then for every  $n\in\mathbb N$ by Lemma~\ref{l1} there exists a continuous mapping $f_n:X\to E$ such that $f_n(x)=r_n(x)$ if $x\in E_n$, and $f_n(x)=x^*$ if  $x\in F_n$.  It is easy to verify that the sequence $(f_n)_{n=1}^\infty$ is pointwise convergent on $X$ and $\lim\limits_{n\to\infty} f_n(X)\subseteq E$. Let $r(x)=\lim\limits_{n\to\infty}f_n(x)$ for all $x\in X$. Then $r(x)=\lim\limits_{n\to\infty}r_n(x)=x$ for all $x\in E$.
\end{proof}

\begin{proposition}\label{cor:ConeOverCountableSet}
  The perfect cone $\Delta_p(X)$ over a countable regular space $X$ is an absolute $B_1$-retract.
\end{proposition}

 \begin{proof}
  We first note that $\Delta_p(X)$ is perfectly normal by Theorem~\ref{th:PropertiesOfCones}~(\ref{PerfectlyNormal}). Assume that $\Delta_p(X)$ is a $G_\delta$-subset of a perfectly normal space  $Z$. Then   $\Delta_p(X)$ is a weak $B_1$-retract of $Z$ by Proposition~\ref{prop:weak}. Moreover, $\Delta_p(X)$ is a contractible $F_\sigma$-subspace of  $Z$. Hence, Theorem~\ref{th:ContrWeakB1} implies that $\Delta_p(X)$ is a $B_1$-retract of~$Z$.
\end{proof}

Let us observe that any $B_1$-retract of a space with a regular $G_\delta$-diagonal is a $G_\delta$-subset of this space~\cite[Proposition 2.2]{KRaex}. But it is not valid for weak $B_1$-retracts as the following example shows.

\begin{example}\label{ex:weakB1}
Let $\mathbb Q$ be the set of all rational numbers and $X=\mathbb Q\cap [0,1]$. Then  $\Delta_p(X)$ is a weak $B_1$-retract of $\mathbb R^2$, but is not a $B_1$-retract of~$\mathbb R^2$.
\end{example}

\begin{proof}
  Indeed, $\Delta_p(X)$ is a weak $B_1$-retract of $\mathbb R^2$ by Proposition~\ref{prop:weak}. Since $\Delta_p(X)$ is not a $G_\delta$-set in $\mathbb R^2$, $\Delta_p(X)$ is not a $B_1$-retract.
\end{proof}

\begin{theorem}\label{th:UnionB1Retracts} Let $X$ be a perfectly normal space, $E$ be a contractible $G_\delta$-subspace of $X$, $x^*\in E$ and let $(E_n:n\in\mathbb N)$ be  a cover of $E$ such that
\begin{enumerate}
\item $E_n\cap E_m=\{x^*\}$ for all $n\ne m$;

  \item $E_n$ is a relatively ambiguous set in $E$ for every $n$;

  \item $E_n$ is a (weak) $B_1$-retract of $X$ for every $n$.
\end{enumerate}
 Then  $E$ is a (weak) $B_1$-retract of  $X$.
\end{theorem}

\begin{proof} From \cite[p.~359]{Ku1} it follows that for every $n$ there exists an ambiguous set $C_n$ in $X$ such that $C_n\cap E=E_n\setminus\{x^*\}$.
Moreover, there exists a sequence $(F_n)_{n=1}^\infty$ of closed subsets of $X$ such that $X\setminus E=\bigcup\limits_{n=1}^\infty F_n$. Let
$D_n=C_n\cup F_n$, $n\ge 1$. Now define $X_1=D_1$ and  $X_n=D_n\setminus (\bigcup\limits_{k<n}  D_k)$ if $n\ge 2$. Then $(X_n:n\in\mathbb N)$ is a partition of $X\setminus\{x^*\}$  by ambiguous sets $X_n$ and $X_n\cap E=E_n\setminus \{x^*\}$ for every $n\ge 1$.

Suppose that $E_n$ is a weak $B_1$-retract of $X$ for every $n$. Choose a sequence $(r_{n,m})_{m=1}^\infty$ of continuous mappings $r_{n,m}:X\to E_n$ such that $\lim\limits_{m\to\infty} r_{n,m}(x)=x$ for all  \mbox{$x\in E_n$}. Since $X_n$ is $F_\sigma$ in $X$, for every $n$ there is an increasing sequence
$(B_{n,m})_{m=1}^\infty$ of closed subsets $B_{n,m}$ of $X$ such that $X_n=\bigcup\limits_{m=1}^\infty B_{n,m}$. Let
$A_{n,m}=\O$ if $n>m$, and $A_{n,m}=B_{n,m}$ if $n\le m$. Then Lemma~\ref{l1} implies that for every $m\in\mathbb N$ there is a continuous mapping $r_m:X\to E$ such that $r_m|_{A_{n,m}}=r_{n,m}$ and $r_m(x^*)=x^*$.

We will show that $\lim\limits_{m\to\infty}r_m(x)=x$ on $E$. Fix $x\in E$. If $x=x^*$ then $r_m(x)=x$ for all $m$. If $x\ne x^*$ then there is a unique $n$ such that $x\in E_n$. Since $(A_{n,m})_{m=1}^\infty$ increases, there exists a number $m_0$ such that $x\in A_{n,m}$ for all $m\ge m_0$. Hence, \mbox{$\lim\limits_{m\to\infty} r_m(x)=\lim\limits_{m\to\infty}r_{n,m}(x)=x$}.
Therefore,  $E$ is a weak $B_1$-retract of $X$.

If $E_n$ is a $B_1$-retract of $X$ for every $n$, we apply similar arguments.
\end{proof}

\section{Cones over ambiguous sets}

\begin{theorem}\label{th:ConeOverOpenSet} Let $\Delta_p(X)$ be the perfect cone over a metrizable locally arcwise connected space $X$, $Z$ be a normal space and let $h:\Delta_p(X)\to Z$ be an embedding such that $h(\Delta_p(X))$ is an ambiguous set in $Z$. If

a) $X$ is separable, or

b) $\Delta_p(X)$ is collectionwise  normal,

\noindent then $h(\Delta_p(X))$ is a $B_1$-retract of $Z$.
\end{theorem}

\begin{proof}
We notice that $h(\Delta_p(X))$ is metrizable, arcwise connected and locally arcwise connected according to Theorem~\ref{th:PropertiesOfCones}.
Then the set $h(\Delta_p(X))$ is a $B_1$-retract of $Z$ by Theorem~\ref{th:Raex1}.
\end{proof}

By $B_\varepsilon(x_0)$ we denote an open ball in a metric space $X$ with center at $x_0\in X$ and with radius $\varepsilon$.

\begin{theorem}\label{th:ConeOverZeroDim}
Let $\Delta_p(X)$ be the perfect cone over a zero-dimensional metrizable separable space $X$, $Z$ be a normal space and let $h:\Delta_p(X)\to Z$ be such a homeomorphic embedding that $h(\Delta_p(X))$ is a closed set in $Z$. Then $h(\Delta_p(X))$ is a weak $B_1$-retract of $Z$.
\end{theorem}

\begin{proof} Without loss of generality we may assume that $\Delta_p(X)$ is a closed subspace of a normal space $Z$. Consider a metric $d$ on $X$ which generates its topological structure and $(X,d)$ is a completely bounded space. For every $n\in\mathbb N$ there exists a finite set $A_n\subseteq X$ such that the family $\mathcal B_n=(B_{\frac 1n}(a):a\in A_n)$ is a cover of $X$. Since $X$ is strongly zero-dimensional~\cite[Theorem 6.2.7]{Eng}, for every $n$ there exists a  finite cover $\mathcal U_n=(U_{i,n}:i\in I_n)$ of $X$ by disjoint clopen sets $U_{i,n}$ which refines~$\mathcal B_n$. Take an arbitrary $x_{i,n}\in U_{i,n}$ for every $n\in\mathbb N$ and $i\in I_n$. For all $x\in X$ and $n\in\mathbb N$ define
 $$
 f_n(x)=x_{i,n},
 $$
if $x\in U_{i,n}$ for some $i\in I_n$. Then every mapping $f_n:X\to X$ is continuous and $\lim\limits_{n\to\infty}f_n(x)=x$ for all $x\in X$.

Fix $n\in\mathbb N$. For all $y\in \Delta_p(X)$ we set
 $$
 g_n(y)=\left\{\begin{array}{ll}
                 \lambda(f_n(\alpha(y)),\beta(y)), & \mbox{if}\,\,\,y\ne v, \\
                 v, & \mbox{if}\,\,\, y=v.
               \end{array}
 \right.
 $$
We prove that $g_n:\Delta_p(X)\to \Delta_p(X)$ is continuous at $y=v$. Indeed, let $(y_m)_{m=1}^\infty$ be a sequence of points $y_m\in Y$ such that $y_m\to v$. Assume that $y_m\ne v$ for all $m$. Show that $g_n(y_m)\to v$. Fix $\varepsilon>0$. Since $\beta(y_m)\to 0$, there is a number $m_0$ such that $\beta(y_m)<\varepsilon$ for all $m\ge m_0$. Then $g_n(y_m)=\lambda(f_n(\alpha(y_m)),\beta(y_m))\in \lambda(X\times [0,\varepsilon))$ for all $m\ge m_0$. Hence,  $g_n$ is continuous at  $v$.

Note that $g_n(\Delta_p(X))\subseteq K_n$, where $K_n=\bigcup\limits_{i\in I_n}vx_{i,n}$. Since $K_n$ is a compact absolute retract by Corollary~\ref{th:ConeOverFiniteSet}, $K_n$ is an absolute extensor. Taking into account that $\Delta_p(X)$ is closed in $Z$, we have that there exists a continuous extension $r_n:Z\to K_n$ of $g_n$.

It remains to show that $\lim\limits_{n\to\infty}r_n(y)=y$ for all $y\in \Delta_p(X)$. Fix $y\in \Delta_p(X)$. If $y=v$ then $r_n(y)=g_n(y)=v$ for all $n\ge 1$. Let $y\ne v$. Since $\lim\limits_{n\to\infty}f_n(\alpha(y))=\alpha(y)$ and $\lambda$ is continuous, $$\lim\limits_{n\to\infty}r_n(y)= \lim\limits_{n\to\infty}\lambda(f_n(\alpha(y)),\beta(y))=\lambda(\alpha(y),\beta(y))=y.$$
Hence, $\Delta_p(X)$ is a weak $B_1$-retract of $Z$.
\end{proof}

\begin{theorem}\label{cor:FsigmaZeroDim}
  The perfect cone $\Delta_p(X)$  over a $\sigma$-compact zero-dimensional metrizable space $X$ is an absolute $B_1$-retract.
\end{theorem}

\begin{proof} Assume that $\Delta_p(X)$ is a $G_\delta$-subspace of a perfectly normal space $Z$.

Since $X$ is $\sigma$-compact, there exists an increasing sequence $(F_n)_{n=1}^\infty$ of compact subsets of $Z$ such that $X=\bigcup\limits_{n=1}^\infty F_n$. Since for every $n\in\mathbb N$ the set $F_{n+1}\setminus F_n$ is open in the zero-dimensional metrizable separable space $F_{n+1}$, there exists a partition $(B_{n,m}:m\in\mathbb N)$ of $F_{n+1}\setminus F_n$ by relatively clopen sets  $B_{n,m}$ in $F_{n+1}$. Let $\mathbb N^2=(n_k,m_k:k\in\mathbb N)$, $H_0=F_1$ and let $H_k=B_{n_k,m_k}$ for every $k\in\mathbb N$.  Then the family  $(H_k:k=0,1,\dots)$ is a partition of $X$ by compact sets $H_k$.

Fix $k\in\mathbb N$. Let $E_k=\Delta_p(H_k)$ be the perfect cone over zero-dimensional metrizable separable space $H_k$. Then $E_k$ is a closed subset of $Z$. Therefore, $E_k$ is a weak $B_1$-retract of $Z$ by Theorem~\ref{th:ConeOverZeroDim}.

Since $\Delta_p(X)=\bigcup\limits_{k=1}^\infty E_k$, Theorem~\ref{th:UnionB1Retracts} implies that $\Delta_p(X)$ is a weak $B_1$-retract of $Z$. It remains to apply Theorem~\ref{th:ContrWeakB1}.
\end{proof}

\begin{theorem}\label{th:coneoverR}
The perfect cone $\Delta_p(X)$  over  a $\sigma$-compact space $X\subseteq \mathbb R$ is an absolute $B_1$-retract.
\end{theorem}

\begin{proof} Suppose that $\Delta_p(X)$ is a $G_\delta$-subspace of a perfectly normal space $Z$. Since $\Delta_p(X)$ is $\sigma$-compact, $\Delta_p(X)$ is $F_\sigma$ in  $Z$.

Let $G={\rm int}_{\mathbb R}X$, $F=X\setminus G$, $A=\Delta_p(G)$ and $B=\Delta_p(F)$.
Since $G$ and $F$ are $\sigma$-compact sets, $A$ and $B$ are $\sigma$-compact sets too. Hence, $A$ and $B$ are ambiguous subsets of $\Delta_p(X)$. Consequently, $A$ and $B$ are ambiguous in $Z$. Since $G$ is metrizable locally arcwise connected separable space, $A$ is a $B_1$-retract of $Z$  by Theorem~\ref{th:ConeOverOpenSet}. Since $F$ is zero-dimensional metrizable $\sigma$-compact space, $B$ is a $B_1$-retract of $Z$ according to Theorem~\ref{cor:FsigmaZeroDim}. Theorem~\ref{th:UnionB1Retracts} implies that the set $\Delta_p(X)=A\cup B$ is a $B_1$-retract of~$Z$.
\end{proof}

Note that the condition of $\sigma$-compactness of $X$ in Theorems~\ref{th:ConeOverZeroDim} and \ref{cor:FsigmaZeroDim} is essential (see Example~\ref{ex:ExMain}).

\section{The weak local connectedness point set of $B_1$-retracts}

Let  $(Y,d)$ be a metric space. A sequence $(f_n)_{n=1}^\infty$ of mappings $f_n:X\to Y$ is {\it uniformly convergent to a mapping  $f$ at a point $x_0$ of $X$} if for any $\varepsilon>0$ there exists a neighborhood $U$ of $x_0$ and $N\in\mathbb N$ such that
$$
d(f_n(x),f(x))<\varepsilon
$$
for all $x\in U$ and $n\ge N$. We observe that if every $f_n$ is continuous at $x_0$ and the sequence $(f_n)_{n=1}^\infty$ converges uniformly to $f$ at $x_0$,
then $f$ is continuous at $x_0$.

By $R((f_n)_{n=1}^\infty,f,X)$ we denote the set of all points of uniform convergence of the sequence $(f_n)_{n=1}^\infty$ to the mapping $f$.

The closure of a set $A$ in a subspace $E$ of a topological space $X$ we denote by $\overline{A}^{\,\,E}$.

A space $X$ is  {\it weakly locally connected at $x_0\in X$} if every open neighborhood of $x_0$ contains a connected (not necessarily open) neighborhood of $x_0$. The set of all points of weak local connectedness of $X$ we will denote by $WLC(X)$.

\begin{theorem}\label{th:LocCon}
  Let $X$ be a locally connected space, $(E,d)$ be a metric subspace of $X$ and let $r:X\to E$ be a $B_1$-retraction which is a pointwise limit of a sequence of continuous mappings $r_n:X\to E$. Then
  $$
  R((r_n)_{n=1}^\infty,r,X)\cap E\subseteq WLC(E).
  $$
\end{theorem}

\begin{proof}
  Fix $x_0\in R((r_n)_{n=1}^\infty,r,X)\cap E$ and  $\varepsilon>0$. Set $W=B_{\varepsilon}(x_0)$.
  Choose a neighborhood $U_1$ of $x_0$ in $X$ and a number $n_0$ such that
  $$
  d(r_n(x),r(x))<\frac \varepsilon 4
  $$
  for all $x\in U_1$ and $n\ge n_0$.
  Since $r$ is continuous at $x_0$, there exists a neighborhood $U_2\subseteq X$ of $x_0$ such that
  $$
  d(r(x),r(x_0))<\frac\varepsilon 4
  $$
  for all $x\in U_2$. The locally connectedness of $X$ implies that there is a connected neighborhood $U$ of $x_0$ such that $U\subseteq U_1\cap U_2$.
  Since $\lim\limits_{n\to\infty}r_n(x_0)=x_0$, there exists a number $n_1$ such that $r_n(x_0)\in U\cap E$ for all $n\ge n_1$. Let
  $N=\max\{n_0,n_1\}$ and
  $$
  F=\overline{\bigcup\limits_{n\ge N}r_n(U)}^{\,\, E}.
  $$
We show that $F\subseteq W$. Let $x\in U$ and $n\ge N$. Then
$$
d(r_n(x),x_0)=d(r_n(x),r(x_0))\le d(r_n(x),r(x))+d(r(x),r(x_0))<\frac\varepsilon 4+\frac\varepsilon 4=\frac \varepsilon 2.
$$
Thus, $r_n(x)\in B_{\varepsilon /2} (x_0)$. Then $\bigcup\limits_{n\ge N}r_n(U)\subseteq B_{\varepsilon /2} (x_0)$. Hence,
$$
F\subseteq \overline{B_{\varepsilon /2} (x_0)}^{\,\, E}\subseteq W.
$$
Moreover, $r(U)\subseteq F$, provided $\lim\limits_{n\to\infty}r_{N+n}(x)=r(x)$ for every $x\in U$. Observe that   $U\cap E=r(U\cap E)\subseteq r(U)$.
Therefore,
$$
x_0\in U\cap E\subseteq F\subseteq W,
$$
which implies that $F$ is a closed neighborhood of $x_0$ in $E$.

It remains to prove that $F$ is a connected set. To obtain a contradiction, assume that $F=F_1\cup F_2$, where $F_1$ and $F_2$ are nonempty disjoint closed subsets of  $F$.
Clearly, $F\cap U\ne\O$.

Consider the case $F_i\cap U\ne\O$ for $i=1,2$. The continuity of $r_n$ implies that $r_n(U)$ is a connected set for every $n\ge 1$. Since $r_n(U)\subseteq F$, $r_n(U)\subseteq F_1$ or $r_n(U)\subseteq F_2$ for every $n\ge N$. Choose $x_i\in F_i\cap U$ for $i=1,2$. Taking into account that $\lim\limits_{n\to\infty}r_n(x_i)=x_i$ for $i=1,2$, we choose a number $k\ge N$ such that $r_n(x_i)\in F_i$ for all $n\ge k$ and for $i=1,2$. Then $r_k(U)\subseteq F_1\cap F_2$, which implies a contradiction.

Now let $F_1\cap U\ne\O$ and $F_2\cap U=\O$. Then $U\cap E\subseteq F_1$. Since $r_n(x_0)\in U\cap E$, $r_n(x_0)\in F_1$, consequently,  $r_n(U)\subseteq F_1$ for all $n\ge N$. Then $F\subseteq \overline{F_1}=F_1$. Therefore, $F_2=\O$, a contradiction.  One can similarly prove that the case when $F_1\cap U=\O$ and $F_2\cap U\ne\O$ is impossible.

Hence, the set $F$ is connected and $x_0\in WLC(E)$.
\end{proof}

Note that we cannot replace the set $R((r_n)_{n=1}^\infty,r,X)$ by the wider set $C(r)$ of all points of continuity of the mapping $r$ in Theorem~\ref{th:LocCon} as the following example shows.

\begin{example}
  There exists an arcwise connected closed subspace $E$ of $\mathbb R^2$ and a $B_1$-retraction $r:\mathbb R^2\to E$ such that $C(r)\cap E\not\subseteq WLC(E)$.
\end{example}

\begin{proof} Let $a_0=(0;0)$, $a_n=(\frac 1n;0)$ for $n\ge 1$ and $X=\{a_n:n=0,1,2,\dots\}$. Denote by $va_n$  the segment which connects the points $v=(1;0)$ and $a_n$ for every $n=0,1,\dots$. Define $E=\bigcup\limits_{n=0}^\infty va_n$.
 Then $E$ is an arcwise connected compact subspace of $\mathbb R^2$ and $WLC(E)=(E\setminus va_0)\cup\{v\}$. For all $x\in \mathbb R^2$ write
   $$
   r(x)=\left\{\begin{array}{ll}
                 x, & \mbox{if}\,\, x\in E, \\
                 a_0, & \mbox{if}\,\, x\not\in E.
               \end{array}
   \right.
   $$
It is easy to see that $r:\mathbb R^2\to E$ is continuous at the point $x=a_0$. We show that $r\in B_1(\mathbb R^2,E)$. Since $X\setminus E$ is $F_\sigma$, choose an increasing sequence of closed subsets $X_n\subseteq \mathbb R^2$ such that $\mathbb R^2\setminus E=\bigcup\limits_{n=1}^\infty X_n$. Let $E_n=\bigcup\limits_{k=0}^n va_k$, $n\ge 1$. For every $n\in\mathbb N$ define $A_n=X_n\cup E_n$. Then for every $n$ the set $A_n$ is closed in $\mathbb R^2$, $A_{n}\subseteq A_{n+1}$ and $\bigcup\limits_{n=1}^\infty A_n=\mathbb R^2$. Clearly, the restriction $r|_{A_n}:A_n\to E_n$ is continuous for every $n$. By the Tietze Extension Theorem there is a continuous extension $f_n:\mathbb R^2\to \mathbb R^2$ of $r|_{A_n}$ for every $n$. Notice that for every $n$ there exists a retraction $\alpha_n:\mathbb R^2\to E_n$. Let $r_n=\alpha_n\circ f_n$. Then $r_n:\mathbb R^2\to E_n$ is a continuous mapping such that $r_n|_{A_n}=r|_{A_n}$ for every~$n$.

It remains to show that $\lim\limits_{n\to\infty}r_n(x)=r(x)$ for all $x\in \mathbb R^2$. Indeed, fix $x\in \mathbb R^2$. Then there is a number $N$ such that $x\in A_n$ for all $n\ge N$. Then $r_n(x)=r(x)$ for all $n\ge N$. Hence, $r\in B_1(\mathbb R^2,E)$.
\end{proof}

\begin{theorem}\label{th:LCnotempty} Let $X$ be a locally connected Baire space and $E$ be a metrizable $B_1$-retract of $X$. Then the set $E\setminus WLC(E)$ is of the first category in $X$.

If, moreover, $X$ has a regular $G_\delta$-diagonal and $E$ is dense in $X$ then $WLC(E)$ is a dense $G_\delta$-subset of $X$.
\end{theorem}

\begin{proof} Let $d$ be a metric on the set $E$ which generates its topological structure. Consider a $B_1$-retraction $r:X\to E$ and choose a sequence $(r_n)_{n=1}^\infty$ of continuous mappings $r_n:X\to E$ such that $\lim\limits_{n\to\infty} r_n(x)=r(x)$  for all $x\in E$. Denote $R=R((r_n)_{n=1}^\infty,r,X)$. Then $R\cap E\subseteq WLC(E)$ by Theorem~\ref{th:LocCon}. According to Osgood's theorem~\cite{Osgood}, $X\setminus R$ is an $F_\sigma$-set of the first category in $X$. Hence, $E\setminus WLC(E)$ is a set of the first category in $X$.

Now assume that $X$ has a regular $G_\delta$-diagonal and $\overline{E}=X$. It follows from \cite[Proposition 2.2]{KRaex} that $E$ is $G_\delta$ in $X$. Moreover, the set $R$ is dense in $X$, since $X$ is Baire. Then $R\cap E$ is dense in $X$. Hence, $WLC(E)$ is dense in $X$.
Observe that $WLC(E)$ is a $G_\delta$-subset of $E$ by~\cite[p.~233]{Ku2}. Then $WLC(E)$ is $G_\delta$ in $X$.
\end{proof}

The following example gives the negative answer to Question~\ref{question}.

\begin{example}\label{ex:ExMain}
  There exists an arcwise connected $G_\delta$-set $E\subseteq \mathbb R^2$ such that
   $E$ is the perfect cone over zero-dimensional metrizable separable space $X\subseteq \mathbb R$ and
  $E$ is not a $B_1$-retract of $\mathbb R^2$.

\end{example}

\begin{proof} Let $\mathbb I$ be the set of irrational numbers and $X=\mathbb I\cap [0,1]$. Define
$$
E=\{(xt,t):x\in X, t\in [0,1]\}.
$$
Then $E\simeq\Delta_p(X)$. Moreover, $E$ is an arcwise connected $G_\delta$-subset of $\mathbb R^2$. Clearly, $\overline{E}=[0,1]^2$ and $WLC(E)=\{v\}$. Therefore, Theorem~\ref{th:LCnotempty} implies that $E$ is not a $B_1$-retract of~$[0,1]^2$. Consequently, $E$ is not a $B_1$-retract of $\mathbb R^2$.
\end{proof}

\section{Acknowledgements} The author would like to thank the referee for his helpful and constructive comments that greatly contributed to improving the final version of the paper.

{\bf e-mail:} maslenizza.ua@gmail.com

{\bf address:} Chernivtsi National University,
                    Department of Mathematical Analysis,
                    Kotsjubyns'koho 2,
                    Chernivtsi 58012,
                    UKRAINE

{\bf keywords:} cone over a space, $B_1$-retract, Baire-one mapping

{\bf AMS Classification:} Primary 54C20, 54C15; Secondary 54C50

\end{document}